\documentclass[twoside,a4paper,11pt]{amsart}

\input xypic
\usepackage{amssymb,enumitem}
\xyoption{all}

\newtheorem{thm}{Theorem}[section]
\newtheorem{lemma}[thm]{Lemma}
\newtheorem{cor}[thm]{Corollary}
\newtheorem{prop}[thm]{Proposition}

\newtheorem{defn}[thm]{Definition}

\def\add{\operatorname{add}\nolimits}

\def\Ext{\operatorname{Ext}\nolimits}
\def\Tor{\operatorname{Tor}\nolimits}
\def\Hom{\operatorname{Hom}\nolimits}
\def\id{\operatorname{id}\nolimits}

\def\rad{\operatorname{rad}\nolimits}
\def\soc{\operatorname{soc}\nolimits}

\def\Mod{\operatorname{Mod-\!}\nolimits}
\def\mod{\operatorname{mod-\!}\nolimits}
\def\Inj{\operatorname{Inj-\!}\nolimits}
\def\pd{\operatorname{pd}\nolimits}
\def\findim{\operatorname{findim}\nolimits}
\def\FinDim{\operatorname{FinDim}\nolimits}
\def\sup{\operatorname{sup}\nolimits}

\author{Jeremy Rickard} 
\address{School of Mathematics, University of Bristol, Bristol BS8 1TW, UK}
\email{j.rickard@bristol.ac.uk}
\date\today 
\title[Unbounded derived categories]{Unbounded derived categories and the finitistic dimension
  conjecture} 
\keywords{ Finite dimensional algebras; derived categories; finitistic dimension conjecture}

\begin{document}
\maketitle

\begin{abstract}
  We consider the question of whether the injective modules generate
  the unbounded derived category of a ring as a triangulated category
  with arbitrary coproducts. We give an example of a non-Noetherian
  commutative ring where they don't, but prove that they do for any
  Noetherian commutative ring. For non-commutative finite dimensional
  algebras the question is open, and we prove that if injectives
  generate for such an algebra, then the finitistic dimension
  conjecture holds for that algebra.
\end{abstract}

\section{Introduction}
\label{sec:intro}

In the representation theory of finite dimensional algebras, a circle
of interconnected questions and conjectures collectively known as the
``homological conjectures'' have inspired a lot of work over the last
few decades. In the 1980s derived categories were introduced into
representation theory, and it was natural to try to interpret these
homological conjectures in terms of derived
categories. Happel~\cite{happel:conj},\cite{happel} related many of
these conjectures to bounded derived categories and
in~\cite{happel:conj} gives a good overview of the conjectures.

More recently, it has increasingly become clear that it is often more
natural and convenient to work with the unbounded derived category of
complexes of arbitrary modules, since this has better properties,
especially having arbitrary coproducts.

In a talk in Nordfjordeid in 2001~\cite{keller}, Keller considered the
question of whether the unbounded derived category is generated, as a
triangulated category with infinite coproducts, by the injective
modules (if this is the case, then we say that ``injectives
generate''), and pointed out that if this were the case then some of
the homological conjectures for finite dimensional algebras would
follow: in particular, the Nunke condition that for any finitely
generated nonzero module $X$ for a finite dimensional algebra $A$,
$\Ext^i_A(DA,X)\neq0$ for some $i\geq0$, where $DA$ is the dual of the
regular module, and \emph{a fortiori} the generalized Nakayama
conjecture that for every simple $A$-module $S$, $\Ext^i_A(DA,S)\neq0$
for some $i\geq0$. He asked whether there was any connection to the
finitistic dimension conjecture, which is known to imply many of the
other homological conjectures, including the Nunke condition and
generalized Nakayama conjecture. 

In this paper we answer this question in Theorem~\ref{thm:findim} by
proving that if injectives generate for a finite dimensional algebra
$A$ then the finitistic dimension conjecture holds for $A$. At present
I know no examples of finite dimensional algebras for which injectives
do not generate.

We first, in Section~\ref{sec:inj}, consider the question of whether
injective modules generate the derived category for more general
rings. Our main results are a proof that they always do for
commutative Noetherian rings, and an example of a commutative
non-Noetherian ring where they don't.

In Section~\ref{sec:colocalizing} we briefly consider the dual condition
of projective modules generating the derived category as a
triangulated category with products. It is not clear that the two
conditions are equivalent, but for finite dimensional algebras we
prove that ``injectives generate'' implies the dual condition.

In Section~\ref{sec:example} we work though an example of a finite
dimensional algebra for which injectives generate to illustrate some
of the more elementary techniques that can be used.

Finally we develop some techniques for proving, for particular finite
dimensional algebras, that injective modules generate the derived
category, and use these to prove that injectives generate for some
classes of algebras where the finitistic dimension conjecture is known
to hold. In some of these, the finitistic dimension conjecture is
obvious, but it is less obvious that injectives generate, and there
are also several classes of algebras for which the finitistic
dimension conjecture is known to hold but for which I don't know
whether injectives generate.

\section*{Acknowledgements}
I would like to thank Birge Huisgen-Zimmermann, Bernhard Keller,
Henning Krause and Greg Stevenson for conversations and email
exchanges related to the contents of this paper.

\section{Preliminaries}
\label{sec:preliminaries}

In this section we shall fix some notation and discuss some
generalities on localizing subcategories.

Let $R$ be a ring, not necessarily commutative. We shall be
considering $R$-modules and complexes of $R$-modules, and our standard
conventions will be that by an $R$-module we mean a right $R$-module,
and by a complex we mean a cochain complex, unless we specify
otherwise. We shall use the following notation for various categories
related to $R$.

\begin{itemize}
\item $\Mod R$ is the category of (right) $R$-modules.
\item $\mathcal{K}(R)=\mathcal{K}(\Mod R)$ is the homotopy category of
  (cochain) complexes of $R$-modules.
\item $\mathcal{D}(R)=\mathcal{D}(\Mod R)$ is the derived category of
  (cochain) complexes of $R$-modules.
\end{itemize}

We impose no finiteness conditions on the modules in $\Mod R$ and no
boundedness conditions on the complexes in $\mathcal{K}(R)$ and
$\mathcal{D}(R)$.

We shall regard $\Mod R$ as a full subcategory of both
$\mathcal{K}(R)$ and $\mathcal{D}(R)$ in the usual way, identifying a
module $M$ with the complex which has $M$ in degree zero and zero in
all other degrees.

The derived category $\mathcal{D}(R)$ is a triangulated category with
arbitrary (small) coproducts, so it makes sense to consider the
localizing subcategory generated by a particular class $\mathcal{S}$
of objects (i.e., the smallest triangulated subcategory of
$\mathcal{D}(R)$ that contains $\mathcal{S}$ and is closed under
coproducts) and to ask whether this is the whole of
$\mathcal{D}(R)$. We shall denote by $\langle\mathcal{S}\rangle$ the
localizing subcategory generated by $\mathcal{S}$, and if this is the
whole of $\mathcal{D}(R)$ then we shall say that ``$\mathcal{S}$
generates $\mathcal{D}(R)$''. There are other notions of generation
for triangulated categories, but this is the only one that we shall
consider here, except briefly in Section~\ref{sec:colocalizing}.

We shall start by considering some easy concrete properties of
localizing subcategories, none of which are original, which will be
useful when we come to look at examples.

\begin{prop}\label{prop:loc}
  Let $\mathcal{C}$ be a localizing subcategory of $\mathcal{D}(R)$.
  \begin{enumerate}[label=\alph*)]
  \item If $0\to X\to Y\to Z\to 0$ is a short exact sequence of
    complexes, and two of the three objects $X$, $Y$ and $Z$ are in
    $\mathcal{C}$, then so is the third.
  \item If a complex $X$ is in $\mathcal{C}$ then so is the shifted
    complex $X[t]$ for every $t\in\mathbb{Z}$.
  \item If $X$ and $Y$ are quasi-isomorphic complexes and $X$ is in
    $\mathcal{C}$, then so is $Y$.
  \item If $\left\{X_i\mid i\in I\right\}$ is a set of objects of
    $\mathcal{C}$, then $\bigoplus_{i\in I}X_i$ is in $\mathcal{C}$.
  \item If $X\oplus Y$ is in $\mathcal{C}$ then so are $X$ and $Y$.
  \item If $X=X^\ast$ is a bounded complex, where the module
    $X^i$ is in $\mathcal{C}$ for every $i$, then $X$ is in
    $\mathcal{C}$.
  \item If
    $X_0\stackrel{\alpha_0}{\longrightarrow}X_1\stackrel{\alpha_1}{\longrightarrow}X_2\stackrel{\alpha_2}{\longrightarrow}\dots$
    is a sequence of cochain maps between complexes, with $X_i$ in
    $\mathcal{C}$ for all $i$, then $\varinjlim X_i$ is in
    $\mathcal{C}$.
  \item If $X:=X^\ast$ is a bounded above complex, where the module
    $X^i$ is in $\mathcal{C}$ for every $i$, then $X$ is in
    $\mathcal{C}$.
  \end{enumerate}
\end{prop}

\begin{proof}
  (a), (b) and (c) are just restatements of the definition of a
  triangulated subcategory, since a short exact sequence of complexes
  becomes a distinguished triangle in the derived category, and (d) is
  the definition of what it means for a triangulated subcategory to be
  localizing.

There is a short exact sequence
$$0\to X\to(X\oplus Y\oplus X\oplus\dots)\to (Y\oplus X\oplus Y\oplus\dots)\to0$$
where the last two terms are both coproducts of countably many copies
of $X\oplus Y$, so (e) follows from (a) and (d). This trick is a form
of the ``Eilenberg swindle''.

Using (a) and (b), an easy induction on the length of a bounded
complex $X^\ast$ proves (f), since if $n$ is maximal with $X^n\neq0$,
then there is an exact sequence of complexes
$$0\to X^n[-n]\to X^\ast\to\sigma^{<n}X^\ast\to0,$$
where $\sigma^{<n}X^\ast=\dots\to X^{n-2}\to X^{n-1}\to 0\to\dots$ is
the ``brutal'' truncation of $X^\ast$.

For (g), there is a short exact sequence
$$0\to\bigoplus_{i=0}^\infty X_i\to\bigoplus_{i=0}^\infty X_i\to\varinjlim X_i\to0,$$
where the first map is $\alpha_n-\id_{X_n}$ on $X_n$, and so the claim
follows from (a) and (d). Note that this is a special case of the fact
that a localizing subcategory is closed under homotopy colimits (see,
for example, \cite{bokstedt_neeman}) of sequences of maps.

A bounded above complex $X^\ast$ is the direct limit of its brutal
truncations
$\sigma^{\geq n}X^\ast=\dots\to0\to X^n\to X^{n+1}\to\dots$, which are
bounded, and so (h) follows from (f) and (g).
\end{proof}

It is important to note that the proof of (h) would not work (and in
fact the statement would be false in general) if we replaced ``bounded
above'' by ``bounded below'', since a bounded below complex is the
inverse limit, rather than direct limit, of its bounded brutal
truncations. It would be true if we also replaced ``localizing
subcategory'' by ``colocalizing subcategory'' (i.e., triangulated
subcategory closed under products).

It is well-known that $\mathcal{D}(R)$ is generated (as a triangulated
category with coproducts) by the projective modules. There are more
sophisticated, and arguably better, ways of proving this, but we'll
give a rather concrete argument, since some of the ideas will be
relevant later on.

\begin{prop}\label{prop:projgen}
  The projective $R$-modules generate $\mathcal{D}(R)$.
\end{prop}

\begin{proof}
  Let $X^\ast$ be a complex of $R$-modules. It is the direct limit of
  its ``good'' truncations
$$\tau^{\leq n}X^\ast=\dots\longrightarrow X^{n-2}\stackrel{d^{n-2}}{\longrightarrow}X^{n-1}\stackrel{d^{n-1}}{\longrightarrow}\ker d^n\longrightarrow0\longrightarrow\dots,$$
which are bounded above, so by (g) and (h) of
Proposition~\ref{prop:loc} it suffices to prove that every $R$-module
is in the localizing subcategory generated by the projectives.

But an $R$-module $M$ is quasi-isomorphic to any projective resolution
of $M$, and is therefore in the localizing subcategory generated by
the projectives by (c) and (h) of Proposition~\ref{prop:loc}.
\end{proof}

Since every projective module is a direct summand of a direct sum of
copies of the regular right $R$-module $R_R$, this shows that the
single object $R_R$ generates $\mathcal{D}(R)$.

\section{The localizing subcategory generated by injectives}
\label{sec:inj}

A similar proof to that of Proposition~\ref{prop:projgen} shows that
the injective $R$-modules generate $\mathcal{D}(R)$ as a triangulated
category with products (i.e., the colocalizing subcategory generated
by the injectives is the whole of $\mathcal{D}(R)$), but we shall
consider the {\em localizing} subcategory generated by
injectives. We'll denote by $\Inj R$ the category of injective
$R$-modules, so this localizing subcategory will be denoted by
$\langle\Inj R\rangle$.

This paper will explore the following condition that a ring might satisfy.

\begin{defn}
  If $\langle\Inj R\rangle=\mathcal{D}(R)$ then we say that {\bf
    injectives generate} for $R$.
\end{defn}

In later sections we shall mainly consider the case of finite
dimensional algebras, but in this section we shall prove a few results
for arbitrary rings and give an example of a ring for which injectives
do not generate.

First, there is one class of rings, including those of finite global
dimension, for which it is very easy to see that injectives generate.

\begin{thm}\label{thm:fingd}
  If the regular right module $R_R$ has finite injective dimension,
  then injectives generate for $R$. In particular, injectives generate
  for any ring with finite global dimension.
\end{thm}

\begin{proof}
  The module $R_R$ is quasi-isomorphic to an injective resolution,
  which if $R$ has finite injective dimension can be taken to be a
  bounded complex of injectives, which is in $\langle\Inj R\rangle$ by
  Proposition~\ref{prop:loc}. Since $R_R$ generates $\mathcal{D}(R)$
  by Proposition~\ref{prop:projgen}, injectives generate.
\end{proof}

For commutative Noetherian rings there is a classification of
localizing subcategories of $\mathcal{D}(R)$ that makes it easy to see
that injectives generate.

\begin{thm}\label{thm:noeth}
  If $R$ is a commutative Noetherian ring, then injectives generate
  for $R$.
\end{thm}

\begin{proof}
  Let $R$ be a commutative Noetherian ring. By a famous theorem of
  Hopkins~\cite{hopkins} and Neeman~\cite{neeman:chromatic}, there is
  an inclusion preserving bijection between the collection of sets of
  prime ideals and the collection of localizing subcategories of
  $\mathcal{D}(R)$, and by Lemma 2.9 of
  Neeman~\cite{neeman:chromatic}, for any prime ideal $\mathfrak{p}$,
  the injective hull $I(R/\mathfrak{p})$ of $R/\mathfrak{p}$ is in the
  localizing subcategory corresponding to $\{\mathfrak{p}\}$, which is
  a minimal nonzero localizing subcategory, and therefore generated by
  $I(R/\mathfrak{p})$. Hence the only localizing subcategory that
  contains all injectives is the one corresponding to
  $\operatorname{Spec}(R)$, the set of all prime ideals, which is
  $\mathcal{D}(R)$.
\end{proof}

Next we show that if $R$ is derived equivalent to a ring for which
injectives generate, then injectives generate for $R$. This follows
from the fact that we can characterize the objects of the derived
category isomorphic to bounded complexes of injectives.

\begin{thm}\label{thm:dereq}
  Let $R$ and $S$ be derived equivalent rings; i.e., $\mathcal{D}(R)$
  and $\mathcal{D}(S)$ are equivalent as triangulated categories. If
  injectives generate for $S$ then injectives generate for $R$.
\end{thm}

\begin{proof}
  It is well-known that the compact objects of the derived category
  are the perfect complexes: those isomorphic to bounded complexes of
  finitely generated projectives.

  The objects with cohomology bounded in degree are those objects $B$
  such that, for each perfect complex $C$, $\Hom(C,B[t])=0$ for all
  but finitely many $t\in\mathbb{Z}$.

  The objects isomorphic to bounded complexes of injectives (i.e.,
  those with bounded injective resolutions) are those objects $X$ such
  that, for every object $B$ with bounded cohomology, $\Hom(B,X[t])=0$
  for all but finitely many $t\in\mathbb{Z}$.

  Thus an equivalence of triangulated categories between
  $\mathcal{D}(R)$ and $\mathcal{D}(S)$ restricts to an equivalence
  between the subcategories of objects isomorphic to bounded complexes
  of injectives. By Proposition~\ref{prop:loc} this subcategory
  generates the same localizing subcategory as the injectives, so if
  injectives generate for one of the two rings $R$ and $S$ then they
  generate for the other.
\end{proof}

We end this section by giving an example of a non-Noetherian
commutative ring for which injectives don't generate.

Let $k$ be a field. If $V$ is a $k$-vector space, $DV$ will denote the
dual space $\Hom_k(V,k)$.

\begin{thm}
  Let $R=k[x_0,x_1,\dots]$, the ring of polynomials over $k$ in
  countably many variables. Injectives do not generate for $R$.
\end{thm}

\begin{proof}
  We denote by $S$ the one-dimensional $R$-module $R/(x_0,x_1,\dots)$,
  and we shall show that $S$ is not in $\langle\Inj R\rangle$.
  
  It suffices to prove that $\Hom(X,S)=0$ for every object $X$ of
  $\langle\Inj R\rangle$, since clearly $\Hom(S,S)\neq0$.

  Since the class of objects $X$ such that $\Hom(X,S[t])=0$ for all
  $t\in\mathbb{Z}$ is a localizing subcategory of $\mathcal{D}(R)$, it
  suffices to prove that, for every injective module $I$ and every
  $t\geq0$, $\Ext^t_R(I,S)=\Hom(I,S[t])=0$ (where $\Ext^0$ denotes
  $\Hom$).
  
  Since $DR$ is an injective cogenerator of the module category, every
  injective $R$-module is a direct summand of some product $DR^J$ of
  copies of $DR$, so it suffices to prove that $\Ext^t_R(DR^J,S)=0$
  for every set $J$ and every $t\geq0$.

  Let $C_i$ be the chain complex
  $\dots\to0\to R\stackrel{x_i}{\to} R\to0\to\dots$ with the second
  non-zero term in degree zero, let $X_n=\bigotimes_{i<n}C_i$, so
  there are natural inclusions $X_1\to X_2\to\dots$, and let
  $X_\infty=\varinjlim_n X_n$. Then the homology of $X_n$ is
  $R/(x_0,\dots,x_{n-1})$ concentrated in degree zero, and so
  $X_\infty$ is a projective resolution of
  $k=\varinjlim_n R/(x_0,\dots,x_{n-1})$.

  Since $S$ is finite dimensional,
  $\Ext^t_R(DR^J,S)=D\Tor_t^R(DR^J,S)$, which is the dual of the
  degree $t$ homology of the chain complex
  $DR^J\otimes_RX_\infty=\varinjlim_n DR^J\otimes_RX_n$. Since $X_n$ is
  a complex of finitely generated $R$-modules,
  $$DR^J\otimes_RX_\infty=\varinjlim\left( (DR\otimes_RX_n)^J\right).$$

  So to prove that $\Ext^t_R(DR^J,k)=0$, it suffices to prove that
  $DR\otimes_RX_n$ has no homology in degree $t$ for sufficiently
  large $n$.

  We shall show that the homology of $DR\otimes_RX_n$ is
  $D\left(R/(x_0,\dots,x_{n-1})\right)$, concentrated in degree $n$.

  This is true for $n=0$. There is a short exact sequence
  $$0\longrightarrow R/(x_0,\dots,x_{m-1})\stackrel{x_m}{\longrightarrow}
  R/(x_0,\dots,x_{m-1})\longrightarrow
  R/(x_0,\dots,x_m)\longrightarrow0,$$ so if it is true for $n=m$ then
  it is true for $n=m+1$.
\end{proof}

There are also non-Noetherian commutative rings for which injectives
do generate. For example, if $R=k[x_0,x_1,\dots]/(x_0,x_1,\dots)^2$ is
the quotient of the ring considered above by quadratic polynomials,
then the methods of Section~\ref{sec:classes} can easily be adapted to
show that injectives generate for $R$.

\section{The finitistic dimension conjecture}
\label{sec:findim}

From now on we specialize to the case of a finite
dimensional algebra $A$ over a field $k$. As at the end of the previous
section, we use the notation $DV$ to denote the $k$-dual of a vector
space $V$. Also we denote by $\mod A$ the full subcategory of $\Mod A$
consisting of finitely generated modules.

One simplifying factor is that if $A$ is a finite dimensional algebra
then every injective $A$-module is a direct summand of a direct sum of
copies of $DA$, so $\langle\Inj A\rangle$ is generated by the single
object $DA$.

We recall the definition of the big and little finitistic dimensions
of a finite dimensional algebra $A$. If $M$ is an $A$-module then
$\pd(M)$ denotes the projective dimension of $M$.

\begin{defn}
  Let $A$ be a finite dimensional algebra.

  The {\bf big finitistic dimension} of $A$ is
  $$\FinDim(A)=\sup\left\{\pd(M)\mid
   M \in\Mod A\mbox{ and }\pd(M)<\infty\right\}.$$

 The {\bf little finitistic dimension} of $A$ is
 $$\findim(A)=\sup\left\{\pd(M)\mid
   M \in\mod A\mbox{ and }\pd(M)<\infty\right\}.$$
\end{defn}

The finitistic dimension conjecture (for finite dimensional algebras)
is the conjecture that $\findim(A)<\infty$. This was publicized as a
question by Bass~\cite{bass} in 1960 and is still open. The
corresponding question for the big finitistic dimension is also open,
and we shall distinguish between the two questions as follows.

\begin{defn}
  Let $A$ be a finite dimensional algebra.

  We say that $A$ {\bf satisfies the big finitistic dimension
    conjecture} if $\FinDim(A)<\infty$.

  We say that $A$ {\bf satisfies the little finitistic dimension conjecture} if
  $\findim(A)<\infty$.
\end{defn}

Clearly $\FinDim(A)\geq\findim(A)$ and so if $A$ satisfies the big
finitistic dimension conjecture then it satisfies the little
finitistic dimension conjecture.

Huisgen-Zimmermann~\cite{huisgen:decades} gives a good
survey of the finitistic dimension conjecture and its history.

\begin{thm}\label{thm:findim}
  Let $A$ be a finite dimensional algebra over a field $k$. If
  injectives generate for $A$ then $A$ satisfies the big finitistic
  dimension conjecture (and hence also the little finitistic dimension
  conjecture).
\end{thm}

\begin{proof}
  Suppose that $A$ does not satisfy the big finitistic dimension
  conjecture. Then there is an infinite family
  $\left\{M_i\mid i\in I\right\}$ of nonzero $A$-modules with
  $\pd(M_i)\neq\pd(M_j)$ for $i\neq j$. Let $d_i=\pd(M_i)$.

  Let $P_i$ be a minimal projective resolution of $M_i$. Then
  $P_i[-d_i]$, considered as a cochain complex, is zero except in
  degrees $0$ to $d_i$, and has cohomology concentrated in degree
  $d_i$.

  Both $\bigoplus_iP_i[-d_i]$ and $\prod_iP_i[-d_i]$ have homology
  isomorphic to $M_i$ in degree $d_i$, and the natural inclusion
$$\iota:\bigoplus_iP_i[-d_i]\to\prod_iP_i[-d_i]$$
is a quasi-isomorphism.

However, $\iota$ is not a homotopy equivalence, since if it were then
applying any additive functor to $\iota$ would give a
quasi-isomorphism. Note that $P_i[-d_i]\otimes_A(A/\!\rad(A))$ has
nonzero cohomology $\Tor^A_{d_i}\left(M_i,A/\!\rad(A)\right)$ in
degree zero. Since taking the tensor product with a finitely presented
module preserves both products and coproducts, the map on degree zero
cohomology induced by $\iota\otimes_A\left(A/\!\rad(A)\right)$ is the
natural map
$$\bigoplus_i\Tor^A_{d_i}\left(M_i,A/\!\rad(A)\right)\to
\prod_i\Tor^A_{d_i}\left(M_i,A/\!\rad(A)\right),$$
which is not an isomorphism.

Thus, if $C$ is the mapping cone of $\iota$, then $C$ is a bounded
below complex of projective $A$-modules that is acyclic but not
contractible.

The functor $-\otimes_ADA$ is an equivalence from the category of
projective $A$-modules to the category of injective
$A$-modules. Applying this functor to $C$, we get a bounded below
complex $C\otimes_ADA$ of injective $A$-modules that is not
contractible (since $-\otimes_ADA$ is an equivalence), and so not
acyclic, since a bounded below complex of injectives is acyclic if and
only if it is contractible.

Since $C$ is acyclic, $\Hom_{\mathcal{K}(A)}(A,C[t])=0$, and so
$$\Hom_{\mathcal{K}(A)}(DA,C\otimes_ADA[t])=0,$$ 
for all $t\in\mathbb{Z}$.

Since $C\otimes_ADA$ is a bounded below complex of injectives,
$$\Hom_{\mathcal{D}(A)}(DA,C\otimes_ADA[t])\cong
\Hom_{\mathcal{K}(A)}(DA,C\otimes_ADA[t])=0.$$ But the objects $X$ of
$\mathcal{D}(A)$ such that
$\Hom_{\mathcal{D}(A)}(X,C\otimes_ADA[t])=0$ for all $t$ form a
localizing subcategory of $\mathcal{D}(A)$, and
therefore $C\otimes_ADA$ is not in $\langle\Inj A\rangle$. So injectives
do not generate for $A$.
\end{proof}

The proof of the last theorem also provides an equivalent form of the
big finitistic dimension conjecture that is similar to an equivalent
form of ``injectives generate'' that follows by a standard application
of Bousfield localization.

We define the {\bf right perpendicular category} $DA^\perp$ of $DA$ to
be the full subcategory of $\mathcal{D}(A)$ consisting of the objects
$X$ such that, for all $t\in\mathbb{Z}$,
$\Hom_{\mathcal{D}(A)}(DA,X[t])=0$. As usual, $\mathcal{D}^+(A)$
denotes the full triangulated subcategory of $\mathcal{D}(A)$
consisting of complexes with cohomology bounded below in degree.

\begin{thm}\label{thm:perp}
Let $A$ be a finite dimensional algebra over a field $k$.
\begin{enumerate}
\item[(a)] Injectives generate for $A$ if and only if
  $DA^\perp=\{0\}$.
\item[(b)] $A$ satisfies the big finitistic dimension conjecture if
  and only if
  $$DA^\perp\cap\mathcal{D}^+(A)=\{0\}.$$
\end{enumerate}
\end{thm}

\begin{proof} If $X$ is an object of $DA^\perp$, then the class
  $^\perp X$ of objects $Y$ for which $\Hom(Y,X[t])=0$ for all
  $t\in\mathbb{Z}$ is a localizing subcategory of $\mathcal{D}(A)$
  containing $DA$, so if injectives generate for $A$ then $X$ is in
  $^\perp X$, and so the existence of the identity map for $X$ shows
  that $X\cong0$. The converse of part (a) is just ``Bousfield
  localization''~\cite{bousfield}. For any object $X$ of
  $\mathcal{D}(A)$, there is a distinguished triangle
$$\Gamma X\to X\to LX\to \Gamma X[1]$$
such that $\Gamma X$ is in $\langle\Inj A\rangle$ and $LX$ is in
$DA^\perp$. If $DA^\perp=\{0\}$ then for every object $X$ of
$\mathcal{D}(A)$,
$X\cong \Gamma X\in\langle\Inj A\rangle$ and so injectives generate for $A$.

That Bousfield localization applies in this situation (or more
generally for a localizing subcategory generated by a single object in
a compactly generated triangulated category) seems to be well-known,
an explicit proof can be found in~\cite[Proposition
5.7]{alonsotarrio_et_al} or, for a more general statement,
in~\cite[Theorem 7.2.1]{krause}.

For (b), the proof of Theorem~\ref{thm:findim} shows that if $A$ does not
satisfy the big finitistic dimension conjecture then there is a
nonzero object of $\mathcal{D}^+(A)$ in $DA^\perp$.

The argument for the converse emerged from a conversation with Henning
Krause. Suppose $X:=X^\ast$ is a nonzero object of
$DA^\perp\cap\mathcal{D}^+(A)$. Without loss of generality, $X$ is a
complex of injectives concentrated in positive degrees and
$X^0\to X^1$ is not a split monomorphism. Since $X$ is a bounded
below complex of injectives,
$$\Hom_{\mathcal{K}(A)}(DA,X[t])\cong\Hom_{\mathcal{D}(A)}(DA,X[t])=0$$
for all $t\in\mathbb{Z}$.

Applying the equivalence $\Hom_A(DA,-)$ between injective and
projective $A$-modules to $X$ we get a complex
$Y^\ast:=\Hom_A(DA,X)$ of projective $A$-modules concentrated in
non-negative degrees such that, for all $t\in\mathbb{Z}$,
$\Hom_{\mathcal{K}(A)}(A,Y^\ast[t])=0$ (i.e., $Y^\ast$ is
acyclic) and such that $Y^0\to Y^1$ is not a split monomorphism.

Hence, for any $t>0$
$$\sigma^{\leq t}Y:= \dots\to 0\to Y^0\to Y^1\to\dots\to Y^t\to0\to\dots$$
is the projective resolution of an $A$-module with projective
dimension $t$, and so $A$ does not satisfy the big finitistic
dimension conjecture.
\end{proof}

\section{The colocalizing subcategory generated by projectives}
\label{sec:colocalizing}

As well as considering whether $\mathcal{D}(R)$ is generated as a
triangulated category with arbitrary small coproducts by the injective
modules, it makes sense to ask the dual question: is $\mathcal{D}(R)$
generated as a triangulated category with arbitrary small products by
the projective modules.

For general rings there seems to be no direct connection between the
two questions, but for finite dimensional algebras there is one
implication.

\begin{prop}
  Let $A$ be a finite dimensional algebra over a field for which
  $\mathcal{D}(A)$ is generated by the injective modules as a
  triangulated category with arbitrary coproducts. Then
  $\mathcal{D}(A^{op})$ is generated by the projective modules as a
  triangulated category with arbitrary products.
\end{prop}

\begin{proof}
  Consider the full subcategory $\mathcal{X}$ of $\mathcal{D}(A)$
  consisting of the objects $X$ for which the dual $DX$ is in the
  colocalizing subcategory of $\mathcal{D}(A^{op})$ generated by
  projectives. Then $\mathcal{X}$ is a localizing subcategory of
  $\mathcal{D}(A)$ that contains $DA$, and so
  $\mathcal{X}=\mathcal{D}(A)$. Hence $DA$ is in the colocalizing
  subcategory of $\mathcal{D}(A^{op})$ generated by projectives, and
  $DA$ generates $\mathcal{D}(A^{op})$ as a triangulated category with
  arbitrary products.
\end{proof}

In fact, the proof of Theorem~\ref{thm:findim} can be adapted to show
that this potentially weaker condition is sufficient to prove the
finitistic dimension conjecture.

\begin{prop}
  Let $A$ be a finite dimensional algebra over a field. If projectives
  generate $\mathcal{D}(A^{op})$ as a triangulated category with
  arbitrary products, then $A$ satisfies the big finitistic dimension
  conjecture.
\end{prop}

We'll only sketch the proof, which is very similar to that of
Theorem~\ref{thm:findim}.

\begin{proof}
  If there is an infinite family $\left\{M_i\mid i\in I\right\}$ of
  nonzero $A$-modules with $\pd(M_i)=d_i\neq d_j=\pd(M_j)$ for
  $i\neq j$ then $\left\{DM_i\mid i\in I\right\}$ is a family of
  nonzero $A^{op}$-modules with $\id(DM_i)=d_i$, so
  $\id(DM_i)\neq\id(DM_j)$ for $i\neq j$.

  Let $I_i$ be a minimal injective resolution of $DM_i$. Then the
  mapping cone $C$ of the natural inclusion
$$\bigoplus_iI_i[d_i]\to\prod_iI_i[d_i]$$
is a bounded above acyclic but not contractible complex of injective
$A^{op}$-modules. Since $C$ is acyclic,
$\Hom_{\mathcal{K}(A^{op})}(C,DA[t])=0$, and so (applying the
equivalence between injective and projective modules)
$$\Hom_{\mathcal{K}(A^{op})}\left(\Hom_{A^{op}}(DA,C), A[t]\right)=0$$
for all $t\in\mathbb{Z}$.

Since $\Hom_{A^{op}}(DA,C)$ is a bounded above complex of projective
$A^{op}$-modules,
$$\Hom_{\mathcal{D}(A^{op})}\left(\Hom_{A^{op}}(DA,C), A[t]\right)=
\Hom_{\mathcal{K}(A^{op})}\left(\Hom_{A^{op}}(DA,C), A[t]\right)=0.$$

But the objects $X$ such that
$\Hom_{\mathcal{D}(A^{op})}\left(\Hom_{A^{op}}(DA,C), X[t]\right)=0$
for all $t\in\mathbb{Z}$ form a colocalizing subcategory of
$\mathcal{D}(A^{op})$, and therefore the colocalizing subcategory
generated by $A$ does not contain $\Hom_{A^{op}}(DA,C)$.
\end{proof}

The only real reason for focusing on generation by injectives in this
paper, rather than the dual condition, is that localizing
subcategories of $\mathcal{D}(A)$ and other triangulated categories
are more familiar than colocalizing subcategories, and have received
more attention.

\section{An example}
\label{sec:example}

In this section we shall prove that injectives generate for a
particular example of a finite dimensional algebra $A$ to illustrate
some of the elementary methods that can be used.

Our strategy will be to explicitly build up a collection of objects of
$\langle\Inj A\rangle$ until we have shown that all the simple
$A$-modules are in $\langle\Inj A\rangle$.

\begin{lemma}\label{lem:simples}
  If $A$ is a finite dimensional algebra, and every simple $A$-module
  is in $\langle\Inj A\rangle$, then injectives generate for $A$.
\end{lemma}

\begin{proof}
  Every semisimple module is a coproduct of simple modules, and
  therefore in the localizing subcategory generated by simple
  modules. So, by induction on the radical length, every module is in
  the localizing subcategory generated by simple modules.
\end{proof}

The algebra $A$ that we shall consider is borrowed from an example in
a paper of Goodearl and Huisgen-Zimmermann~\cite{goodearl_huisgen}
that we shall discuss further in the next section. There is no
particular motive behind choosing this algebra for our example, and in
fact, owing to differences in left/right conventions, we actually
consider the opposite algebra to the one that they studied.

$A$ is a binomial relation algebra: the quotient $A=kQ/I$ of the path
algebra of a quiver $Q$ (with six vertices labelled $1,\dots,6$) by an
admissible ideal $I$ generated by some paths and some elements of the
form $p-\lambda q$, where $p$ and $q$ are paths and $\lambda$ is a
nonzero scalar. For vertex $i$, the corresponding idempotent (trivial
path) of $A$ will be denoted by $e_i$, the indecomposable projective
$e_iA$ by $P_i$, the indecomposable injective $D(Ae_i)$ by $I_i$, and
the simple module $P_i/\rad(P_i)\cong\soc(I_i)$ by $S_i$ .

We use diagrams to describe modules in a similar way
to~\cite{goodearl_huisgen}. The indecomposable projective modules
$P_1,\dots,P_6$ are described by the following diagrams.

$$
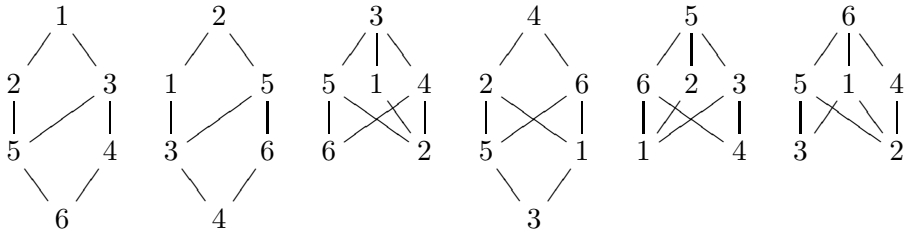
\xymatrix@R12pt@C6pt{
&1\ar@{-}[ld]\ar@{-}[rd]&\\
2\ar@{-}[d]&&3\ar@{-}[lld]\ar@{-}[d]\\
5\ar@{-}[rd]&&4\ar@{-}[ld]\\
&6&\\
}\quad
\xymatrix@R12pt@C6pt{
&2\ar@{-}[ld]\ar@{-}[rd]&\\
1\ar@{-}[d]&&5\ar@{-}[lld]\ar@{-}[d]\\
3\ar@{-}[rd]&&6\ar@{-}[ld]\\
&4&\\
}\quad
\xymatrix@R12pt@C6pt{
&3\ar@{-}[ld]\ar@{-}[d]\ar@{-}[rd]&\\
5\ar@{-}[d]\ar@{-}[rrd]&1\ar@{-}[rd]&4\ar@{-}[lld]\ar@{-}[d]\\
6&&2\\
}\quad
\xymatrix@R12pt@C6pt{
&4\ar@{-}[ld]\ar@{-}[rd]&\\
2\ar@{-}[d]\ar@{-}[rrd]&&6\ar@{-}[lld]\ar@{-}[d]\\
5\ar@{-}[rd]&&1\ar@{-}[ld]\\
&3&\\
}\quad
\xymatrix@R12pt@C6pt{
&5\ar@{-}[ld]\ar@{-}[d]\ar@{-}[rd]&\\
6\ar@{-}[d]\ar@{-}[rrd]&2\ar@{-}[ld]&3\ar@{-}[lld]\ar@{-}[d]\\
1&&4\\
}\quad
\xymatrix@R12pt@C6pt{
&6\ar@{-}[ld]\ar@{-}[d]\ar@{-}[rd]&\\
5\ar@{-}[d]\ar@{-}[rrd]&1\ar@{-}[ld]\ar@{-}[rd]&4\ar@{-}[d]\\
3&&2\\
}
$$

The numbers represent the elements of a basis for the module, each of
which is annihilated by all but one of the idempotents $e_i$, that one
being the one with the same number. Lines joining two numbers indicate
that an arrow of $Q$ maps the basis element nearer the top of the
diagram to a nonzero scalar multiple of the basis element nearer the
bottom. For all pairs $(i,j)$ of vertices the quiver $Q$ has at most
one arrow from vertex $i$ to vertex $j$, so there is never any
ambiguity about which arrow is meant. The absence of a line indicates
that a corresponding arrow acts as zero. So the top two layers of the
diagrams for the indecomposable projectives tells us what the quiver
is (e.g., there are two arrows starting at vertex $1$, one of which
ends at vertex $2$ and one at vertex $3$), and the rest of the diagram
determines the relations apart from the precise values of the scalars
$\lambda$, which we shall not need to know.

Not every module can be represented by such a diagram, but all the
modules we need to consider can be, and in a way that determines the
module up to isomorphism.

The projectives $P_1=I_6$, $P_2=I_4$ and $P_4=I_3$ are also injective,
and the other indecomposable injectives $I_1$, $I_2$ and $I_5$ are
described by the following diagrams.

$$\xymatrix@R12pt@C6pt{
4\ar@{-}[d]\ar@{-}[rd]&&5\ar@{-}[lld]\ar@{-}[ld]\ar@{-}[d]\\
6\ar@{-}[rd]&2\ar@{-}[d]&3\ar@{-}[ld]\\
&1&\\
}\quad
\xymatrix@R12pt@C6pt{
3\ar@{-}[d]\ar@{-}[rd]\ar@{-}[rrd]&&6\ar@{-}[lld]\ar@{-}[ld]\ar@{-}[d]\\
5\ar@{-}[rd]&1\ar@{-}[d]&4\ar@{-}[ld]\\
&2&\\
}\quad
\xymatrix@R12pt@C6pt{
1\ar@{-}[d]\ar@{-}[rd]&&4\ar@{-}[ld]\ar@{-}[d]\\
3\ar@{-}[rd]&2\ar@{-}[d]&6\ar@{-}[ld]\\
&5&\\
}
$$

Let
$$M_a=\vcenter{\xymatrix@R12pt@C6pt{
4\ar@{-}[d]\\
6\\
}}\quad\quad\mbox{and}\quad\quad
M_b=\vcenter{\xymatrix@R12pt@C6pt{
&1\ar@{-}[ld]\ar@{-}[rd]&\\
2\ar@{-}[rd]&&3\ar@{-}[ld]\\
&5&\\
}}$$
so that there are short exact sequences
$$0\to M_a\to I_6\to M_b\to0$$
and $$0\to M_b\to I_5\to M_a\to0.$$

Splicing copies of these short exact sequences, we get exact sequences
$$\dots\to I_6\to I_5\to I_6\to I_5\to M_a\to 0$$
and
$$\dots\to I_5\to I_6\to I_5\to I_6\to M_b\to 0,$$
so that $M_a$ and $M_b$ are quasi-isomorphic to bounded above
complexes of injectives and are therefore in $\langle\Inj A\rangle$ by
Proposition~\ref{prop:loc}(h).

Now let
$$M_c=\vcenter{\xymatrix@R12pt@C6pt{
&2\ar@{-}[ld]\ar@{-}[rd]&\\
1\ar@{-}[rd]&&5\ar@{-}[ld]\\
&3&\\
}}$$
so that there is a short exact sequence
$$0\to M_c\to I_3\to M_a\to0.$$
Since the second and third terms are in $\langle\Inj A\rangle$, $M_c$
is also in $\langle\Inj A\rangle$ by Proposition~\ref{prop:loc}(a).

Let
$$M_d=\vcenter{\xymatrix@R12pt@C6pt{
6\ar@{-}[d]\\
4\\
}}$$
so that there is a short exact sequence
$$0\to M_d\to I_4\to M_c\to0,$$
proving that $M_d$ is in $\langle\Inj A\rangle$.

There are exact sequences
$$\dots\to M_d\to M_a\to M_d\to M_a\to S_4\to0$$
and
$$\dots\to M_a\to M_d\to M_a\to M_d\to S_6\to0,$$
so $S_4$ and $S_6$ are also in $\langle\Inj A\rangle$.

Let
$$M_e=\vcenter{\xymatrix@R12pt@C6pt{
6\ar@{-}[rd]&&3\ar@{-}[ld]\\
&1&\\
}}\quad\quad\mbox{and}\quad\quad
M_f=\vcenter{\xymatrix@R12pt@C6pt{
4\ar@{-}[rd]&&5\ar@{-}[ld]\\
&2&\\
}}$$

There are short exact sequences
$$0\to M_e\to I_1\to M_f\to0$$
and
$$0\to M_f\to I_2\to M_e\to0,$$
so as above we can deduce that $M_e$ and $M_f$ are in
$\langle\Inj A\rangle$, and since we already know that $S_4$ and $S_6$
are in $\langle\Inj A\rangle$ we can deduce that
$$M_g=\vcenter{\xymatrix@R12pt@C6pt{
3\ar@{-}[d]\\
1\\
}}\quad\quad\mbox{and}\quad\quad
M_h=\vcenter{\xymatrix@R12pt@C6pt{
5\ar@{-}[d]\\
2\\
}}$$
are in $\langle\Inj A\rangle$.

Let
$$M_i=\vcenter{\xymatrix@R12pt@C6pt{
2\ar@{-}[d]\\
5\ar@{-}[d]\\
6\\
}}\quad\quad\mbox{and}\quad\quad
M_j=\vcenter{\xymatrix@R12pt@C6pt{
1\ar@{-}[d]\\
3\ar@{-}[d]\\
4\\
}}$$
The short exact sequences
$$0\to M_i\to I_6\to M_j\to0$$
and
$$0\to M_j\to I_4\to M_i\to0$$
allow us to deduce that $M_i$ and $M_j$, and hence also
$$M_k=\vcenter{\xymatrix@R12pt@C6pt{
    2\ar@{-}[d]\\
    5\\
  }}\quad\quad\mbox{and}\quad\quad M_l=\vcenter{\xymatrix@R12pt@C6pt{
    1\ar@{-}[d]\\
    3\\
  }}$$ are in $\langle\Inj A\rangle$. Finally $S_1$ is in
$\langle\Inj A\rangle$ because of the exact sequence
$$\dots\to M_g\to M_l\to M_g\to M_l\to S_1\to 0,$$
and similarly the remaining simple modules $S_2$, $S_3$ and $S_5$ are
in $\langle\Inj A\rangle$.

\section{Some methods to show that injectives generate }
\label{sec:classes}

In this section we shall describe some methods to prove, for
particular algebras and certain classes of algebras, that injectives
generate. This includes classes that have long been known to satisfy
the finitistic dimension conjecture, including some classes where the
finitistic dimension conjecture is obvious, but the fact that
injectives generate is less so. There are other classes of algebras
which are known to satisfy the finitistic dimension conjecture (again,
sometimes for obvious reasons), for which I do not know whether
injectives generate.

As in Section~\ref{sec:example} the strategy that we shall generally
adopt is to show that all the simple modules are in the localizing
subcategory generated by injectives, by giving explicit constructions.

We shall frequently be considering syzygies, and more particularly
cosyzygies of modules, so we start with some notation. Throughout this
section, $A$ will always be a finite dimensional algebra over a field.

Let $M$ be an $A$-module. For $i\geq0$, the $i$th syzygy $\Omega^iM$ of
$M$ is the kernel of $d_{i-1}$ in a minimal projective resolution
$$\dots\stackrel{d_3}{\longrightarrow}P_2\stackrel{d_2}{\longrightarrow}P_1
\stackrel{d_1}{\longrightarrow}P_0\stackrel{d_0}{\longrightarrow}M
\stackrel{d_{-1}}{\longrightarrow}0$$
of $M$, and the $i$th cosyzygy $\Sigma^iM$ is the kernel of
$d^i$ in a minimal injective resolution
$$0\longrightarrow M\longrightarrow I^0\stackrel{d^0}{\longrightarrow} 
I^1\stackrel{d^1}{\longrightarrow} I^2\longrightarrow\dots$$ of $M$.

The importance of syzygies and cosyzygies to the finitistic dimension
conjecture has long been recognized. We shall borrow some terminology
and ideas from Goodearl and
Huisgen-Zimmermann~\cite{goodearl_huisgen}, which itself builds on
previous ideas of Jans~\cite{jans}, Colby and
Fuller~\cite{colby_fuller}, Igusa and Zacharia~\cite{igusa_zacharia},
Kirkman, Kuzmanovich and Small~\cite{kirkman_kuzmanovich_small},
Wilson~\cite{wilson} and others. Goodearl and Huisgen-Zimmermann work
primarily with syzygies, often for the opposite algebra, but for the
connection with the localizing subcategory generated by injectives it
is more convenient for us to use cosyzygies, so our definitions will
be roughly dual to theirs. When considering syzygies and cosyzygies of
finitely generated modules, which will usually be the case, they will
be precisely dual.

We shall use the standard notation that $\add(X)$, for a module $X$,
means the class of modules that are direct summands of finite direct
sums of copies of $M$, and more generally, if $\mathcal{X}$ is a
collection of modules, then $\add(\mathcal{X})$ is the class of
modules that are direct summands of finite direct sums of copies of
modules from $\mathcal{X}$.

\begin{defn}
  Let $M$ be a finitely generated $A$-module. We say that $M$ has {\bf
    finite cosyzygy type} if there is a finite set $\mathcal{X}$ of
  indecomposable $A$-modules such that, for every $t\geq0$, $\Sigma^tM$
  is in $\add(\mathcal{X})$.
\end{defn}

\begin{prop}\label{prop:cosyzygy-type}
  Suppose that a finitely generated module $M$ for a finite
  dimensional algebra $A$ has finite cosyzygy type. Then $M$ is in
  $\langle\Inj A\rangle$.
\end{prop}

\begin{proof}
  For $n\geq0$ let $\mathcal{X}_n$ be the class of indecomposable
  $A$-modules that occur as a direct summand of $\Sigma^tM$ for some
  $t\geq n$. Since $M$ has finite cosyzygy type, $\mathcal{X}_0$
  contains only finitely many isomorphism classes. Also
$$\dots\subseteq\mathcal{X}_{n+1}\subseteq\mathcal{X}_n\subseteq\dots\subseteq\mathcal{X}_0,$$ 
so there is some $n$ for which $\mathcal{X}_m=\mathcal{X}_n$ for all
$m\geq n$.

Thus every module in $\add(\mathcal{X}_n)$ occurs as a direct summand
of $\Sigma^mM$ for infinitely many $m\geq n$, and so if $X$ is the
direct sum of countably many modules from each isomorphism class of
indecomposable modules in $\add(\mathcal{X}_n)$, then
$$\bigoplus_{m\geq n}\Sigma^mM\cong X\cong\bigoplus_{m>n}\Sigma^mM.$$

Taking the coproduct of the short exact sequences
$$0\to \Sigma^mM\to I^m\to\Sigma^{m+1}M\to0$$
for $m\geq n$ gives a short exact sequence
$$0\to X\to I\to X\to0,$$
where $I$ is injective.

Splicing together infinitely many copies of this short exact sequence
gives an exact sequence
$$\dots\to I\to I\to I\to X\to 0\to\dots.$$

Hence $X$ is quasi-isomorphic to a bounded above complex of
injectives, and so $X$ is in $\langle\Inj A\rangle$ by
Proposition~\ref{prop:loc}(h).

By Proposition~\ref{prop:loc}(e), $\Sigma^nM$ is in
$\langle\Inj A\rangle$, since it is a direct summand of $X$.

Since $M$ is quasi-isomorphic to the complex
$$\dots\to 0\to I^0\to I^1\to\dots\to I^{n-1}\to\Sigma^nM\to0\to\dots,$$
which is in $\langle\Inj A\rangle$ by Proposition~\ref{prop:loc}(g),
$M$ is in $\langle\Inj A\rangle$ by Proposition~\ref{prop:loc}(c).
\end{proof}

\begin{cor}\label{cor:simples-cosyzygy}
  If the simple $A$-modules have finite cosyzygy type, then injectives
  generate for $A$.
\end{cor}

\begin{proof}
  By Proposition~\ref{prop:cosyzygy-type} every simple $A$-module is in
  $\langle\Inj A\rangle$, and so by Lemma~\ref{lem:simples},
  injectives generate for $A$.
\end{proof}

There are some classes of algebras for which this condition is
obviously satisfied.

\begin{cor}\label{cor:finrep}
  Injectives generate for any finite dimensional algebra $A$ with
  finite representation type.
\end{cor}

\begin{proof}
  Clearly Corollary~\ref{cor:simples-cosyzygy} applies.
\end{proof}

\begin{cor}\label{cor:rad2}
  Injectives generate for any finite dimensional algebra $A$ with
  $\rad^2(A)=0$.
\end{cor}

\begin{proof}
  If $S$ is a simple $A$-module, then $\Sigma^iS$ is semisimple for
  every $i$, and so Corollary~\ref{cor:simples-cosyzygy} applies.
\end{proof}

Recall that a monomial algebra is the path algebra of a quiver modulo
an admissible ideal generated by paths.

\begin{cor}\label{cor:monomial}
  Injectives generate for a monomial algebra $A$.
\end{cor}

\begin{proof}
  The opposite algebra of a monomial algebra is also a monomial
  algebra, and the simple modules for a monomial algebra are syzygy
  finite~\cite{igusa_zacharia,huisgen_domino}. So
  Corollary~\ref{cor:simples-cosyzygy} applies.
\end{proof}

This implies the fact, which was first proved by Green, Kirkman and
Kuzmanovich~\cite{green_kirkman_kuzmanovich}, that every monomial
algebra satisfies the finitistic dimension conjecture.

Here is another definition that is simply the dual of a definition of
Goodearl and Huisgen-Zimmermann~\cite{goodearl_huisgen}.

\begin{defn}
  Let $M$ be a finitely generated $A$-module. If there is some
  $n\geq0$ such that every non-injective indecomposable direct summand
  of $\Sigma^nM$ occurs as a direct summand of $\Sigma^tM$ for
  infinitely many $t$, then the {\bf cosyzygy repetition index} of $M$
  is the least such $n$. If there is no such $n$ then the cosyzygy
  repetition index of $M$ is $\infty$.
\end{defn}

Clearly a module with finite cosyzygy type also has finite cosyzygy
repetition index, since we could take the $n$ in the definition of
cosyzygy repetition index to be the $n$ that occurs in the proof
of Proposition~\ref{prop:cosyzygy-type}.

\begin{prop}\label{prop:repetition}
  If a finitely generated module $M$ for a finite dimensional algebra
  $A$ has finite cosyzygy repetition index, then $M$ is in
  $\langle\Inj A\rangle$.
\end{prop}

\begin{proof}
  We shall suppose the cosyzygy repetition index of $M$ is $n$.

  The proof is almost identical to that of
  Proposition~\ref{prop:cosyzygy-type}, except that we'll take the
  coproduct of the short exact sequences
$$0\to\Sigma^mM\to I^m\to\Sigma^{m+1}M\to0$$
for $m$ {\em strictly} greater than $n$ in order to get a short exact sequence
$$0\to X\to I\to X\to 0.$$
\end{proof}

One useful fact that Goodearl and
Huisgen-Zimmermann~\cite{goodearl_huisgen} prove is that the big
finitistic dimension conjecture is true for $A$ if every simple module
for the opposite algebra is a submodule of a module with finite syzygy
type (or more generally with finite syzygy repetition index). There
seems to be no obvious reason that a similar condition should imply
that injectives generate for $A$ since, for example, this condition is
obviously satisfied for local algebras (for which the finitistic
dimension is also obvious since there are no non-projective modules
with finite projective dimension), but proving that injectives
generate for local algebras seems no easier than for general finite
dimensional algebras.

However, there is a way of extracting information from the localizing
subcategory generated by injectives, short of proving that it is the
whole derived category, in order to prove that an algebra satisfies
the big finitistic dimension conjecture, and this does generalize
results of~\cite{goodearl_huisgen}.  

\begin{thm}
  Let $A$ be a finite dimensional algebra over a field. Suppose that,
  for every simple $A$-module $S$, there is an object $J_S$ of
  $\langle\Inj A\rangle$ with a map $\alpha_S: J_S\to S$ in
  $\mathcal{D}(A)$ that is non-zero on cohomology (for example, this
  would be the case if there were a module $J_S$ in
  $\langle\Inj A\rangle$ with a copy of $S$ in its head). Then $A$
  satisfies the big finitistic dimension conjecture.
\end{thm}

\begin{proof}
  Suppose that $A$ does not satisfy the big finitistic dimension
  conjecture. Then by Theorem~\ref{thm:perp} there is a nonzero object
  $X$ of $\mathcal{D}^+(A)$ in $DA^\perp$. Replacing $X$ by an
  injective resolution and shifting in degree, we can assume that $X$
  is a complex of injective modules concentrated in positive degrees,
  with nonzero cohomology in degree zero.

  Let $S$ be a simple submodule of $H^0(X)$, so there is a map
  $\beta:S\to X$ that is nonzero on cohomology, and let
  $\alpha_S:J_S\to S$ be the map guaranteed by the hypotheses of the
  theorem. Then the composition $\beta\alpha_S: J_S\to X$ is nonzero
  on cohomology, and therefore is a nonzero map in $\mathcal{D}(A)$
  from an object of $\langle\Inj A\rangle$ to an object of $DA^\perp$,
  giving a contradiction.
\end{proof}

\section{Summary and open questions}
\label{sec:summary}

I know of no examples of finite dimensional algebras for which it is
known that injectives do not generate.

If it were proven that injectives generate for all finite dimensional
algebras, then in light of Theorem~\ref{thm:findim} some longstanding
open problems would be settled. And perhaps an example where injectives
do not generate might give a clue that would help find a
counterexample to some other conjectures.

For convenience, the following theorem summarizes various classes of
finite dimensional algebras for which the results of this paper prove
that injectives do generate. The list is not complete (we omit some of
the more technical conditions that follow from
Section~\ref{sec:classes}) or efficient (some of the classes properly
contain others).

\begin{thm}\label{thm:summary}
  Let $A$ be a finite dimensional algebra over a field. If $A$ is in
  any of the following classes of algebras, then injectives generate
  for $A$.
  \begin{enumerate}
  \item[(a)] Commutative.
  \item[(b)] Gorenstein.
  \item[(c)] Finite global dimension.
  \item[(d)] Self-injective.
  \item[(e)] Finite representation type.
  \item[(f)] Radical square zero.
  \item[(g)] Monomial.
  \item[(h)] Algebras derived equivalent to any of the above.
  \end{enumerate}
\end{thm}

\begin{proof}
  Finite dimensional algebras are Noetherian, so (a) follows from
  Theorem~\ref{thm:noeth}. An algebra is Gorenstein if projective
  modules have finite injective dimension and injective modules have
  finite projective dimension (whether one of the conditions implies
  the other is an open question), so (b) follows from
  Theorem~\ref{thm:fingd}. Both (c) and (d) are special cases of
  (b). Corollaries~\ref{cor:finrep},~\ref{cor:rad2}
  and~\ref{cor:monomial} give (e), (f) and
  (g). Theorem~\ref{thm:dereq} implies (h).
\end{proof}

There are other classes of finite dimensional algebras for which the
finitistic dimension conjecture is known to hold, but for which I
don't know whether injectives generate. For example, local algebras
satisfy the finitistic dimension conjecture for obvious reasons (there
are no non-projective modules with finite projective dimension), but
it doesn't seem any easier to prove that injectives generate for local
algebras than it is for general algebras.

The little finitistic dimension conjecture has been proved for
algebras with radical cube zero by Green and
Huisgen-Zimmermann~\cite{green_huisgen} and for algebras with
representation dimension three by Igusa and
Todorov~\cite{igusa_todorov}. However, the big finitistic dimension
conjecture does not seem to be known for these algebras, so maybe we
should expect that proving that injectives generate should be harder
than for other classes of algebras where the big finitistic dimension
conjecture is known to hold.

Even if it turns out to be hard to find an example of a finite
dimensional algebra for which injectives don't generate, it might be
easier to find a right Noetherian ring. By Theorem~\ref{thm:noeth}
such an example would need to be noncommutative.

It would also be interesting to clarify the connection, briefly
touched on in Section~\ref{sec:colocalizing}, between injectives
generating the derived category as a localizing subcategory and
projectives generating as a colocalizing subcategory, both for finite
dimensional algebras and for general rings.

\bibliography{mybib}{} 
\bibliographystyle{amsalpha}


\end{document}